\documentclass[10pt]{amsart}
\usepackage{enumerate}
\usepackage{graphicx}
\usepackage{amsfonts}
\usepackage{graphicx}
\usepackage{amsmath}
\usepackage{amssymb}
\usepackage{amscd}

     \usepackage{enumerate}
     \usepackage{exscale}
     \usepackage{a4wide}
     \usepackage{amsmath,amsfonts,amssymb}
     \newcommand{\be}{\begin{equation}}
     \newcommand{\ee}{\end{equation}}
     \newcommand{\bea}{\begin{eqnarray*}}
     \newcommand{\eea}{\end{eqnarray*}}
     \newcommand{\beq}{\begin{eqnarray}}
     \newcommand{\eeq}{\end{eqnarray}}
     \newcommand{\CC}{\mathbb{C}}

     \newcommand{\NN}{\mathbb{N}}

     \newcommand{\RR}{\mathbb{R}}
     \newcommand{\ZZ}{\mathbb{Z}}

     \newcommand{\cD}{\mathcal{D}}
     \newcommand{\tcD}{\mathcal{\tilde D}}
     \newcommand{\cE}{\mathcal{E}}

     \newcommand{\cH}{\mathcal{H}}
     \newcommand{\cM}{\mathcal{M}}

     \newcommand{\cV}{\mathcal{V}}

     \newcommand{\supp}{\mathop{\mathrm{supp}}}
     \newcommand{\Id}{{\mathop{\mathrm{Id}}}}
     \newcommand{\id}{{\mathop{\mathrm{id}}}}
     \newcommand{\Tr}{\mathop{\mathrm{Tr}}}

     \newcommand{\kom}[1]{}
     
     \DeclareMathOperator*{\ran}{\mathrm{ran}}
     
     \DeclareMathOperator*{\dens}{\mathrm{dens}}
     \DeclareMathOperator*{\dom}{\mathrm{dom	}}
     \newtheorem{thm}{Theorem}[section]
     \newtheorem{lem}[thm]{Lemma}
     \newtheorem{prp}[thm]{Proposition}
     \newtheorem{cor}[thm]{Corollary}
     \theoremstyle{definition}
     \newtheorem{dfn}[thm]{Definition}
     
     \theoremstyle{remark}
     \newtheorem*{rem}{Remark}
     \newtheorem{Rem}[thm]{Remark}

\newcommand{\Hm}[1]{\leavevmode{\marginpar{\tiny%
$\hbox to 0mm{\hspace*{-0.5mm}$\leftarrow$\hss}%
\vcenter{\vrule depth 0.1mm height 0.1mm width \the\marginparwidth}%
\hbox to 0mm{\hss$\rightarrow$\hspace*{-0.5mm}}$\\\relax\raggedright #1}}}

\begin{document}

\title[Jumps of the integrated density of states]
{Hamiltonians on discrete structures: Jumps of the integrated density of states and uniform convergence}

\author{Daniel Lenz}
\address[D. Lenz]{Friedrich-Schiller-Universit\"at Jena, Fakult\"at f\"ur Mathematik \& Informatik,
Mathematisches Institut, 07737 Jena, Germany}
\email{daniel.lenz@uni-jena.de}
\urladdr{http://www.tu-chemnitz.de/mathematik/mathematische\_physik}

\author{Ivan Veseli\'c}
\address[I.Veseli\'c]{Emmy-Noether-Programme of the Deutsche Forschungsgemeinschaft\vspace*{-0.2cm} }
\address{\& Fakult\"at f\"ur Mathematik,\, 09107\, TU\, Chemnitz, Germany}
\urladdr{www.tu-chemnitz.de/mathematik/enp/}

\subjclass[2000]{}

\begin{abstract}
We study equivariant families of discrete  Hamiltonians on amenable geometries and their integrated density of states (IDS).   We prove that
the eigenspace of a fixed energy is spanned by eigenfunctions with compact support.  The size of a jump of the  IDS
is  consequently given by the  equivariant dimension of the subspace spanned by such eigenfunctions.
From this we deduce uniform convergence (w.r.t.~the spectral parameter) of the finite volume
approximants of the IDS.
Our framework includes quasiperiodic operators on Delone sets,  periodic and random operators
on quasi-transitive graphs,  and  operators on percolation graphs.
\end{abstract}

\thanks{{\today, \jobname.tex}} 
\maketitle

\section{Introduction}

The integrated density of states, in the following abbreviated as IDS,  can be defined for a variety of models,
ranging from Hamiltonians in the quantum theory of solids to Laplacians on $p$-cochains on CW-complexes.
We refer e.g.~to \cite{Pastur-73,Shubin-78,KunzS-80,KirschM-82c,AvronS-83,Bellissard-86,Lueck-02,Veselic-06b,KirschM-07} which represent but a  fraction of the literature devoted to this
topic. While some of these references concern operators on continuum configuration space, in the present paper we
restrict ourselves to models on discrete spaces.
Under certain geometric
conditions the IDS can be approximated
by its analogues associated to finite volume restrictions of the Hamiltonian.
Here the approximation is a priori understood in the sense of weak convergence
of measures. We show that under an amenability condition the following
properties are universal for a wide range of models:

\begin{enumerate}[(a)]
 \item If $\lambda$ is a point of discontinuity of the IDS, there exist compactly supported eigenfunctions to $\lambda$,
 and these compactly supported  eigenfunctions actually span the whole eigenspace of $\lambda$.
  \item The IDS can be approximated by its finite volume analogues \emph{uniformly}  in the spectral variable, i.e. with respect to the supremum norm.
 \item The size of each jump of the IDS can be approximated by the jumps of the finite volume analogues.
\end{enumerate}

Some of our models are random, i.e.~concern a whole family of operators.
For such models the three properties above hold \emph{almost surely}.
We present our results in a  general setting. They can be applied to
a variety of models considered before, for instance, Anderson and
quantum percolation models \cite{Pastur-80,DelyonS-84,ChayesCFST-86,Veselic-05b,KirschM-06,AntunovicV-b},
quasi-crystal Hamiltonians on Delone sets \cite{Hof-93,Hof-95, LenzS-02}, Laplacians
on $p$-cochains on complexes \cite{DodziukM-97,DodziukM-98,Eckmann-99}, Harper operators \cite{Shubin-94,MathaiY-02,MathaiSY-03}, random hopping models \cite{KloppN-03}, and
Hamiltonians associated to percolation on tilings \cite{Hof-98b,MuellerR}.

The geometric framework which allows us to treat all these models at once
is given by an action of an amenable group $\Gamma$ on a metric space $X$
such that the two are roughly isometric. It is not surprising that the notion of rough isometry is
fitting in this context since it has been proposed in \cite{Gromov-93} as a tool for studying
geometric properties of a space `at infinity'. On the other hand, it is well known that the IDS does not change
under compactly supported  perturbations. Due to the chosen setting we can
treat in parallel situations where the underlying group is continuous and discrete.

There is no model known to us where all of the features (a) -- (c) were obtained before.
Partial results, however,  were known for some  specific models mentioned above.
Let us now  briefly discuss these earlier partial results recovered in our framework.
Results about pointwise convergence  were obtained in, e.g.,
\cite{Hof-93,Pastur-73,MathaiY-02,MathaiSY-03}.
Note that in the literature special attention was devoted
to convergence at the bottom of the spectrum \cite{DodziukM-97,DodziukM-98,Eckmann-99},
since it corresponds to approximation of Betti-numbers.
Results on uniform convergence of the IDS for models with
finite local complexity were obtained in \cite{LenzS-06,LenzMV,Elek-a}.
Statement (a) was established for periodic models related to abelian groups in \cite{Kuchment-91,Kuchment-05}.
 Weaker  statements about the characterisation of the jumps of the IDS were  subsequently  proven in \cite{KlassertLS-03} and  \cite{Veselic-05b}.  Approximation of the jumps of the IDS  by the finite volume analogues are  contained in \cite{MathaiSY-03}.
Let us stress that jumps of the IDS actually do occur for several models --- like
quasi-periodic models \cite{KlassertLS-03} and percolation Hamiltonians \cite{ChayesCFST-86,Veselic-05b}.
Thus uniform convergence is by no means  automatically implied by pointwise convergence.

Our theorems show that uniform convergence of the IDS is an universal phenomenon and does not depend on
specific features of the model, like number-theoretic properties or finite local complexity.
We hope that this clarification will be helpful for the study of finer properties of the IDS,  which may be indeed model-dependent.
Among those are: the quantisation of jumps of the IDS \cite{Donnelly-81,Sunada-88},
their location \cite{MathaiY-02,DodziukLMSY-03,Veselic-05b}, and the low energy behaviour of the IDS.
Note that this  low energy behaviour has been a recent focus of attention in particular  for  percolation models
(see e.g.~\cite{KirschM-06,MuellerS,AntunovicV-b}) while surveys of results for other models can be e.g.~found in
\cite{Lueck-02,KirschM-07}.

In the context of uniform convergence, we would also  like to emphasize that our method of proof works without a uniform ergodic theorem. This is a fundamental difference  to the earlier considerations of \cite{LenzS-06, LenzMV} (see \cite{Lenz-02} as well).
On the other hand the question whether such an ergodic theorem holds in the present contexts remains  open.
We consider this an interesting question.

This paper is organized as follows: In Section \ref{s-Setting} we
present our results and fix the notation.  Section \ref{Consequences}
provides the necessary background information on rough isometries. The
next four sections are devoted to the proofs of our three main
results. Section \ref{s-Aperiodic} discusses how results on aperiodic order fit into our
framework. In Section \ref{PeriodicExamples} we show how earlier results on
periodic operators on graphs can be recovered by our approach.  Section \ref{s-APH} discusses percolation models and
finally, Section \ref{s-DelonePercolation}  is devoted to models related to percolation on Delone sets.

\section{Setting and results}\label{s-Setting}
Let $(X,d)$ be a locally compact metric space with a countable basis
of the topology.  Let $\Gamma$ be a locally compact amenable
unimodular group with an invariant metric $d_\Gamma$ such that every
ball is precompact.  The invariant Haar measure of a Borel mesurable set $A \subset \Gamma$ is denoted by $|A|$.
Let $\Gamma$ act continuously by  isometries  on
$X$ such that the following two properties hold:

\begin{itemize}
\item There exists a fundamental domain $F'$ with compact closure $F$, which
  is a countable union of compact sets.
\item The map  $\Phi\colon X\longrightarrow \Gamma, x\mapsto \gamma$,  whenever $x\in
\gamma F'$, is a \emph{rough isometry} i.e.~there exist $ b\geq 0$ and
$a\geq 1$ with
$$ \frac{1}{a} d_\Gamma (\Phi(x),\Phi(y)) - b \leq d (x,y) \leq a
d_\Gamma (\Phi(x),\Phi(y)) +b$$ for all $x,y\in X$.
\end{itemize}

Note that all these assumptions are automatically satisfied if both $\Gamma$ and $X$ are discrete
and $\Gamma$ acts cocompactly and freely on $X$.
Models in such a geometric setting are considered in Sections \ref{PeriodicExamples} and \ref{s-APH} below.

Our operators will be families of operators indexed by elements of a
certain topological space. This topological space will be considered
next (see Appendix for details).  We consider the following family of
uniformly discrete subsets of $X$
$$\mathcal{D} :=\{ A\subset X\mid d(x,y)\geq 1,\:\;\mbox{for }\;x,y\in
A\;\:\mbox{with}\;x\neq y\}.$$ The lower bound $d(x,y)\geq 1$ could be
replaced by any other positive number.  Whenever $Y$ is a compact
space we can then equip the set of functions $f : A\longrightarrow Y$,
where $A\in \mathcal{D}$, with the vague topology and obtain a compact
space $\cD_Y$. In fact, there is a choice of $Y$ which is in some sense universal,
and which we discuss in the appendix. There we show in particular that the space
\[
\tcD:= \{(A, h)\mid A \in \cD, h \colon A \times A\to \CC^*  \}.
\]
is naturally equipped with the vague topolgy and moreover compact.
Here $\CC^*$ is an arbitrary compactification of $\CC$.
The
mapping $(A,h) \mapsto \sum_{x, y \in A} g(x,y)h(x,y) \in\CC$ is
continuous for any  continuous $g\colon X\times X \to \CC$ with compact support.
In particular, it is measurable for any continuous  $g\colon X\times X \to \CC$.
This measurablility then holds for as well for any nonnegative measurable $g\colon X\times X \to \CC$.

The action of $\Gamma$ on $X$ induces an action on $\tcD$ by $\gamma (A, h):= (\gamma A, \gamma h)$
where $\gamma A :=\{\gamma a\mid a\in A\}$ and  $(\gamma h)(x,y) = h(\gamma^{-1}x,\gamma^{-1}y)$
for $ x,y \in \gamma A$.
Let $\mu$ be an $\Gamma$-invariant probability measure on $\tcD$, whose topological support we denote by $\Omega$.
Then, $\Omega$ is a closed and hence compact subset of $\tcD$.
For $\omega = (A,h) \in \tcD$  we use the notation $X(\omega):= A$
for the projection on the first component. The latter is a discrete metric subspace of $X$ and gives
rise to $\ell^2 (X(\omega))$ with the natural counting measure
$\delta_{X(\omega)}:=\sum_{x \in X(\omega)} \delta_x$.
For $\omega\in\tcD$ and $\Lambda \subset X$ we denote by
$\Lambda(\omega)$ the intersection  $\Lambda \cap X(\omega)$ and by $\omega(\Lambda)$ the cardinality
$|\Lambda (\omega)|=\delta_{X(\omega)}(\Lambda)$.

The measure $\mu$ induces a $\Gamma$-invariant measure $m$ on $X$ by
\[
m(\Lambda) := \int_\Omega \omega (\Lambda) d\mu(\omega).
\]
For any $S < \infty$, let $M_S$ be the maximal number of points with mutual
distance at least $1$ contained in a ball of radius $S$ in $X$. Then $M_S$ is finite for all $S$ by the very definition of $\cD$.
Thus, the measure $m$ is bounded by $M_S$ on a ball of size $S$ and hence $m$ is  finite on bounded sets.
 We assume that $\Omega$ does not consist of the empty set only. This implies $m(X)\neq 0$,
since the vague topology on $\Omega$ is Hausdorff.
As $F'$ is a countable union of compact sets, the set $I F'$ is measurable for
any measurable $I\subset \Gamma$ by standard monotone class arguments.
The map $\Gamma \supset I  \mapsto m(I F')$ gives
an invariant measure on $\Gamma$. By uniqueness of the Haar measure we
infer that $  \text{dens}(m)  := \frac{m (I F')}{|I|}$ is a nonnegative constant and hence
\[
 m(I F') = \text{dens}(m) |I|
\]
for all $I\subset \Gamma$ compact with $|I| \neq 0$.
Since $m(X ) >0$, $\dens (m)$ is actually strictly positive.
A direct calculation using unimodularity now shows that $u:=
\frac{1}{|I|} \chi_{I F'}$ satisfies
\begin{equation}\label{e-admissible}
 u \text{ has compact support and } 1 = \int u (\gamma^{-1} x) d\gamma \text{ for all $x$}  \in X
\end{equation}
for any $I\subset \Gamma$ compact with $|I|>0$ (see Proposition  \ref{Huhu}).

\medskip

Denote by $\cH$ the direct integral Hilbert space $\int_\Omega^\oplus
d\mu(\omega) \, \ell^2(X(\omega))$.
Each $\omega = (X(\omega), h)\in \tcD$ gives rise to an operator
\[
 H_\omega  \colon c_0(X(\omega)) \to \ell^2(X(\omega)), \quad
(H_\omega v) (x) :=\sum_{y\in X(\omega)}H_\omega(x,y) v (y), \quad \text{ with }  H_\omega(x,y):=h(x,y)
\]
where $c_0(X(\omega))$ denotes the complex valued functions with compact support in $X(\omega)$.
If $H_\omega \colon\ell^2(X(\omega)) \to \ell^2(X(\omega))$
is bounded  by $C \in \RR$ for all $\omega\in \Omega$, we call
\[
H \colon \cH \to \cH, \quad
 H = \int_\Omega^\oplus d\mu (\omega) \, H_\omega
\]
a \emph{bounded decomposable operator}. Note that in this case
the values of $|H_\omega(x,y)|$ are bounded by $C$ for all $\omega\in \Omega$ and $x,y\in X(\omega)$.
A decomposable operator is called of \emph{finite hopping
range} if there exists an $R<\infty$ such that for all $\omega \in
\Omega$ and $x,y \in X(\omega)$, $d (x,y) \ge R$ implies
$H_\omega(x,y) =0$.

Now, let for every $\gamma\in \Gamma$ a function
$s_\gamma \colon X\longrightarrow \{z\in \CC: |z|=1\}$ be given.  A
decomposable operator is called \emph{equivariant} (w.r.t. the family
$s_\gamma$) if and only if
$$ {s_\gamma (x)} H_{\gamma \omega} (\gamma x,\gamma y)
\overline{s_\gamma (y)} = H_\omega (x,y)$$ for all $\gamma\in \Gamma$,
$\omega\in \Omega$, and $x,y\in X(\omega)$. Of course, this is equivalent
to
$$ H T_\gamma = T_\gamma H,$$ for all $\gamma \in \Gamma$ with the
unitary map $T_\gamma \colon\mathcal{H}\longrightarrow \mathcal{H}$,
$(T_\gamma f)_\omega (x) = s_{\gamma^{-1} } (x) f_{\gamma^{-1} \omega}
(\gamma^{-1} x)$.
\begin{rem}
It turns out that for the results and proofs of this paper
it does not matter whether $s_\gamma$ is a non-trivial function or $s_\gamma \equiv 1$.
The reason is that all calculations concern the diagonal matrix elements either
of one single operator $H$ or of a product of two operators $HK$.
In this case, the terms $s_\gamma (x) $ and $\overline{s_\gamma (x)} $ cancel and
equivariance of operators $H$ and $K$  gives
$$H_{\gamma \omega} (\gamma x,\gamma x) = H_\omega (x,x)$$
and
$$ (H K)_{\gamma \omega} (\gamma x, \gamma x) = \sum_{y\in X(\omega)} H_\omega (x,y) K_\omega (y,x).$$
Thus the function $s_\gamma$ can be immediately eliminated.
 \end{rem}

We are interested in operators $H$ satisfying (A) given as follows:
\begin{itemize}
\item[(A)] $H \colon \cH \to \cH$  is a  bounded, selfadjoint, decomposable,
  equivariant operator $H$  of finite hopping range.
\end{itemize}

As usual we denote the spectral family of a selfadjoint operator $T$ by $E_T$.

\begin{thm}  \label{main1}
Let $H$ satisfy (A).  Let $u$ satisfy \eqref{e-admissible}. Then, there
exist a unique measure $\nu_H $ on $\RR$ with
\begin{equation}
\label{e-main1}
\nu_H (\varphi) =  \frac{1}{\dens(m)}  \int_\Omega \Tr( u \varphi (H_\omega)) d\mu(\omega)
\end{equation}
for all $\varphi \in C_c (\RR)$. The measure $\nu_H$ does
not depend on $u$ provided \eqref{e-admissible} is satisfied.  It is a
spectral measure for $H$, i.e. $\nu_H (B) =0$ if and only if $E_H (B) =0$.
\end{thm}

The distribution function of $\nu_H$ is denoted by $N_H$, i.e.
 $$N_H \colon \RR\longrightarrow \RR,\; N_H(\lambda):= \nu_H ((-\infty,\lambda]),$$
and called the \emph{integrated density of states}, IDS for short.

\begin{thm} \label{main2} Let $H$ satisfy (A).  Let $\lambda\in \RR$ be
  arbitrary. Let $P_{comp} := (P_{comp,\omega})$, where $P_{comp,
 \omega}$ is the projection onto the subspace of $\ell^2 (X(\omega))$
 spanned by compactly supported solutions of $(H_\omega -\lambda)
 u=0$. Then,
\begin{equation}
\label{e-equality}
E_H (\{\lambda\})= P_{comp}.
\end{equation}
\end{thm}


The theorem does not assume ergodicity. If ergodicity holds,
then the statement can be strengthened to give the following corollary.

\begin{cor}\label{Korollar} Let $H$ satisfy (A).  Let $\mu$ be ergodic. Let $\lambda\in \RR$ be arbitrary.
Then, the following assertions are equivalent.
\begin{itemize}
\item[(i)] $N_H$ is discontinuous at $\lambda$.
\item[(ii)] For a set of positive measure $\Omega' \subset \Omega$
there  exists a compactly supported nontrivial solution $u \in \ell^2 (X(\omega))$ of
 $(H_\omega -\lambda) u =0$ for each $\omega \in \Omega'$.
\item[(iii)] For almost every $\omega\in\Omega$ the space of $\ell^2$ solutions $u$ to   $(H_\omega -\lambda) u  =0$ is spanned by compactly supported nontrivial  solutions.
\end{itemize}
\end{cor}

\medskip

Since $\Gamma$ is by assumption amenable, there exists a F\o lner sequence
 in $\Gamma$. This is by definition a sequence of compact, nonempty
sets $(I_n)_n, I_n \subset \Gamma$ such that for any compact $K \subset \Gamma$
and $\epsilon >0$ we have $| I_n \triangle K I_n| < \epsilon|I_n|$ if $n$ is large enough.
By passing to a  subsequence one may assume that $(I_n)_n$ is tempered, i.e. that
for some $ C \in (0,\infty)$ and all $n \in \mathbb{N}$ the condition
$| \bigcup_{k <n} I_k^{-1}I_n| \le C | I_n|$ holds.

Let $(I_n)$ be a tempered  {F\o lner} sequence in $\Gamma$ and define $\Lambda_n := I_n F$.
Let $H_{\omega,n} $ be the restriction of $H_\omega$ to $\ell^2
(\Lambda_n(\omega))$ i.e.
\[
H_{\omega,n} =p_{\Lambda_n (\omega)}    H_\omega  \, i_{\Lambda_n (\omega)}
\]
with the natural projection $p_{\Lambda_n (\omega)}  \colon \ell^2 (X(\omega)) \longrightarrow \ell^2
(\Lambda_n (\omega))$ and the natural injection $i_{\Lambda_n (\omega)}\colon \ell^2
(\Lambda_n (\omega)) \longrightarrow \ell^2 (X(\omega))$.  Note that the
space $\ell^2 (\Lambda_n (\omega))$ is finite dimensional.  Let
$N_{\omega,n}$ be the corresponding eigenvalue counting function, i.e.
\[
N_{\omega,n} (\lambda):=\Tr E_{H_{\omega,n}}( (-\infty,\lambda])
= \dim\ran \; E_{H_{\omega,n}}( (-\infty,\lambda]).
\]
Let $\nu_{\omega,n}$ be the measure whose distribution function is $N_{\omega,n}$.

\begin{thm}\label{main3} Let $H$ satisfy (A). Assume that $\mu$ is ergodic.
Then,  the sequence
$\frac{N_{\omega,n}}{\omega(\Lambda_n)}$ converges with respect to the
supremum norm $\|\cdot\|_\infty$ to $N_H$ for $\mu$-almost every $\omega\in \Omega$,
\end{thm}

Even without the assumption of ergodicity we obtain
uniform convergence of the distribution functions
$\frac{N_{\omega,n}}{\omega(\Lambda_n)}$, but the limit cannot be described explicitely,
see Remark \ref{r-non-ergodic} for details.

\section{Rough isometry and linear algebra}\label{Consequences}
In this section, we collect some basic facts about rough isometries and subspace dimensions.
They provide our working tools for the subsequent sections.

\medskip

Let the assumptions of the previous section hold.
 Let $p\in X$ be fixed with $\Phi (p) = \id \in \Gamma$.
For $\Lambda\subset X$,  $x\in X$ and $r>0$   we define
$ d(x,\Lambda):=\inf\{ d (x,y)\mid y\in \Lambda\}$ and
\[
\Lambda^r:=\{x\in X\mid d (x,\Lambda) < r\}, \
\Lambda_r :=\{x\in X: d (x,X\setminus \Lambda) > r\}, \
\partial^r \Lambda :=\Lambda^r\setminus \Lambda_r.
\]
For $I\subset \Gamma$ we use the analogous notation with $d$ replaced by $d_\Gamma$.
Denote by $B_s(\gamma)$ the open ball around  $\gamma\in \Gamma$ with radius $s$.
If $\gamma$ is the identity we simply write $B_s$ for $B_s(\gamma)$.

\begin{lem} For $r\geq 0$  and any  $s \geq ar +b$, the inclusion  $F^r\subset B_{s } F'$ holds.
\end{lem}
\begin{proof}
For any $x\in F^r$ there exists $y  \in F'$ with $d(x,y) <r$
and a unique $\gamma \in \Gamma$ such that $x \in \gamma F'$.
Thus $d_\Gamma (\gamma,\id) \le a d(x,y) + b < a r + b$, implying $\gamma \in B_{s}$.
\end{proof}

\begin{prp}\label{comparable}
For each $r>0$,  let $q\geq   ar +ab +b$ and $s \geq a r + b$  be given.  Then
for all $I\subset \Gamma$:
$$
(I F)^r \subset I^{s } F'\;\:\mbox{and}\:\; (I_{q } F)\subset (I F')_r
\;\:\mbox{and}\:\;  \partial^r (I F)\subset (\partial^{q} I) F'  .
$$
\end{prp}
\begin{proof}
Since $\Gamma$ acts on $X$ by isometries, $(\gamma F)^r =\gamma F^r$.
By the previous Lemma $(I F)^r =I F^r\subset I B_{s} F'=I^{s} F'$  and the first inclusion holds.
For $x \in I_q F'$ there exist unique $\gamma \in I_q, x_0 \in F'$ such that $ x= \gamma x_0$.
By definition $d_\Gamma (\gamma, \beta) > q$ for all $\beta \in \Gamma\setminus I$.
Let $y \in X\setminus IF'$ be arbitrary. Then there are unique $y_0 \in F', \alpha\in \Gamma\setminus I$
with $y = \alpha y_0$. Consequently
$ d(x,y)=d(\gamma x_0,\alpha y_0)  \ge  d_\Gamma (\gamma, \alpha)/a - b > q/a -b$ and thus $ x \in (IF')_r$
for $ r = q/a -b$. Since  $I_{ar +ab + b} F \subset I_{ar +ab + b} B_b F'
\subset I_{ar +ab}  F' \subset (I F')_r$  we proved the second inclusion.
The last inclusion is a combination of the previous two inclusions.
\end{proof}

Recall that $M_S$ denotes the maximal number of points with distance at least one contained in a ball of radius $S$ in $X$.
\begin{prp}\label{UniversalBound}
For any $\rho>0$ and $\displaystyle C :=\frac{M_{2a(\rho +b)}}{|B_\rho|}$
\[
\omega (I F') \leq C\, |I^\rho|,
\]
for all $I\subset \Gamma $ and all $\omega\in \mathcal{D}$.
\end{prp}
Recall that for $I$ precompact $|I^\rho|$ is finite.
\begin{proof}
Choose $S>ab$ and  define $N_X(S,A)$ to be the maximal number of points in $A\subset X$
with distance at least $S$ between them  and  $N_\Gamma (S,I)$
to be the maximal number of points in $I\subset \Gamma$ with distance at least $S$
between them. Then,
\[
 \omega (I F') \leq N_X(1,IF') \leq M_{2S}\, N_X(S,IF') \leq M_{2S}\, N_\Gamma (S/a -b,I).
\]
Set $\rho:=S/a -b>0$. Since
$
|B_\rho| \, N_\Gamma (\rho,I)= \big |  \bigcup_{i=1}^{N_\Gamma (\rho,I)} B_\rho(\gamma_i)\,\big| \le |I^\rho|
$
we obtain the claim.
\end{proof}

The proposition has the following consequence, which is crucial for our results.

\begin{prp}\label{boundary}
Let $(I_n)$ be a {F\o lner} sequence. Then for arbitrary $r\ge 0$ and $\omega \in \mathcal{D}$,
\begin{equation*}
\lim_{n\to \infty} \frac{\omega ((\partial^r I_n) F')}{|I_n|} =0
\text{ and }
\lim_{n\to \infty} \frac{\omega (\partial^r \Lambda_n)}{|I_n|} =0.
 \end{equation*}
\end{prp}
\begin{proof}
Proposition \ref{UniversalBound} gives us $
 \omega ((\partial^r I) F') \le \frac{M_{a\rho+b}}{|B_\rho|} |(\partial^r I)^\rho|
\le \frac{M_{a\rho+b}}{|B_\rho|} |(\partial^{r+\rho} I)| $, thus the {F\o lner} property
of $(I_n)$ implies the first equality. Since $r>0$ was arbitrary and
$\partial^r (I F)\subset (\partial^{q} I) F$, for $q \ge ar +ab+b$,
the second equality follows immediately.
\end{proof}

\begin{lem}\label{LA} For $\omega\in \mathcal{D}$  and $\Lambda\subset X$ let
  $U$ be a subspace of $\ell^2 (\Lambda (\omega))$. For $S\geq 0$ denote by
  $U_S$ the subspace consisting of all functions in $U$ which  vanish outside of $\Lambda_S$.
Then,
$$ 0 \le \dim( U) -     \dim (U_S)\leq \omega ( \partial^S \Lambda).$$
\end{lem}
\begin{proof} Let $Q \colon U \longrightarrow \ell^2 \big((\Lambda\setminus\Lambda_S)(\omega )\big)$ be the natural restriction map. Then,
$$ \dim(U) -\dim(\ker Q)= \dim(\ran Q) .$$
As $\ker Q = U_S$ and $ \dim(\ran Q) \leq \omega (\Lambda\setminus \Lambda_S )\leq  \omega ( \partial^S \Lambda)$, the statement follows.
\end{proof}

\section{Covariant operators and their trace}
In this section we briefly discuss  covariant operators and their natural trace. We  will then  provide a proof of Theorem \ref{main1}.

\medskip

Let $\mathcal{N}$ be the set of all covariant decomposable bounded operators.
Obviously, $\mathcal{N}$ is a vector space in the natural way. Moreover, by
$T^\ast:=(T_\omega^\ast)$ and $ TS :=( T_\omega S_\omega)$ it becomes a
$\ast$-algebra. In fact, it is even a von Neumann algebra.  We will not use this
in the sequel. We will use that for any selfadjoint $T=(T_\omega) \in \mathcal{N}$
the operator
\[
 f (T) = ( f(T_\omega)) \ \text{ for a bounded and measurable }\ f \colon \RR \to \CC
\]
defined by spectral calculus is an element of $\mathcal{N}$ as well.
We start by having a look at functions which satisfy
\eqref{e-admissible}.

\begin{prp}\label{Huhu}
For every  measurable precompact $I\subset \Gamma$ with $0 < |I|$, the function
 $|I|^{-1} \chi_{I F'}$ satisfies \eqref{e-admissible}.
\end{prp}
\begin{proof}  As $I$ is precompact and the action of $\Gamma$ is continuous by assumption, the function $\chi_{I F'}$ has compact support.
Let $x\in X$ be arbitrary. Then, $x$ can be uniquely written as $x =
\alpha x_0$ with $\alpha \in \Gamma$ and $x_0\in F'$. A short
calculation then shows $\chi_{I F'}(\gamma^{-1} x) = \chi_{\alpha
I^{-1}} (\gamma)$ and the claim follows easily from unimodularity.
\end{proof}

For a function $u$ on $X$ and $I\subset \Gamma$ we define $u_I $ by
$u_I (x) :=\int_I u (\gamma^{-1} x) d\gamma$.

\begin{prp} If $u$ satisfies \eqref{e-admissible}, so does $|I|^{-1} u_I$ for any
 precompact measurable   $I\subset \Gamma$ with $0 < |I|$.
\end{prp}
\begin{proof} This follows by a direct calculation using Fubini theorem and  unimodularity.
\end{proof}

\begin{thm}\label{tracetheorem} Let $u$ satisfy \eqref{e-admissible}. Then, the map

$$\tau\colon \mathcal{N} \longrightarrow \RR, \; \tau (T) := \frac{1}{\dens(m)}  \int \Tr (u T_\omega) d\mu(\omega)$$
does not depend on $u$ and satisfies the following properties:
\begin{itemize}
\item $\tau$ is faithful, i.e. for $T\geq 0$, $\tau (T)=0$ if and only if $T=0$.
\item $\tau$ has the trace property, i.e. $\tau (ST) = \tau (T S)$ for all $T,S\in \mathcal{N}$.
\end{itemize}
\end{thm}
\begin{proof} This can essentially be obtained from the groupoid theoretical considerations in
  \cite{LenzPV-02}.  For the convenience of the reader we give a direct proof (see \cite{LenzS-03b,LenzS-03c}
for similar calculations as well).

\medskip

We first show that  $\tau$ does not depend on $u$. Let $u$ and $v$ satisfying
  \eqref{e-admissible} be given. Then,
\begin{eqnarray*} \int_\Omega \sum_{x\in X(\omega)} u(x) T_\omega (x,x) d\mu (\omega)
  &=& \int_\Omega \sum_{x\in X(\omega)} u(x) T_\omega (x,x) \left( \int_\Gamma
  v(\gamma^{-1} x) d\gamma\right) d\mu (\omega)\\
(\mbox{Fubini})\; &=& \int_\Omega \int_\Gamma \left(\sum_{x\in X(\omega)} u(x) T_\omega
  (x,x) v(\gamma^{-1} x) \right) d\gamma d\mu (\omega)\\
(\mbox{covariance}) \:&=& \int_\Omega \int_\Gamma \left(\sum_{y\in \gamma^{-1}\omega}
  u(\gamma y)  T_{\gamma^{-1} \omega} (y,y) v(y)\right)  d\gamma d\mu
  (\omega)\\
(\mu \;\mbox{invariant})\:\; &=& \int_\Omega \int_\Gamma \sum_{y\in X(\omega)}  u (\gamma y) T_\omega (y,y)
  v(y)d \gamma d\mu (\omega)\\
(\mbox{Fubini})\; &=& \int_\Omega \sum_{y\in X(\omega)}  \left(\int_\Gamma u (\gamma
  y)d\gamma\right)  T_\omega (y,y) v(y)  d\mu (\omega)\\
&=& \int_\Omega \sum_{y\in X(\omega)} v(y) T_\omega (y,y) d\mu (\omega).
\end{eqnarray*}
and independence is proven.  By independence of $u$ and Proposition \ref{Huhu} we can replace $u$ by $
\frac{1}{|I|}\chi_{I F'}$ for any precompact $I \subset \Gamma$.
If $T\ge 0$ and $\tau(T)=0$  we have $\chi_{B_r F'} T_\omega \chi_{B_r F'}=0$
for almost all $\omega$ and any $r >0$. Since the sequence $B_r F'$ exhausts $X$,
$T$ must be zero. Thus faithfulness is proven.
We finally show $\tau ( T K) = \tau (K T)$. The calculation is
similar to the one to show independence of $u$. The definition of
$\tau$ gives after inserting $1=\int_\Gamma u (\gamma^{-1} y) d\gamma$ and
using Fubini
\begin{eqnarray*}
\tau (T K) &=& \int_\Omega \int_\Gamma \left(\sum_{x,y\in X(\omega)} T_\omega (x,y)
K_\omega (y,x) u(x) u(\gamma^{-1} y) \right)d\gamma d\mu (\omega)\\
\mbox{(covariance)}\; &=& \int_\Omega \int_\Gamma \left(\sum_{x,y\in \gamma^{-1}
    \omega} T_{\gamma^{-1} \omega} (x,y)
K_{\gamma^{-1} \omega} (y,x) u(\gamma x) u( y) \right) d\gamma d\mu (\omega)\\
(\mu \;\mbox{invariant})\:\; &=& \int_\Omega \int_\Gamma  \left(\sum_{x,y\in X(\omega)} T_{\omega} (x,y)
K_{ \omega} (y,x) u(\gamma x) u( y) \right) d\gamma d\mu (\omega)\\
&=& \tau ( K T).
\end{eqnarray*}
This finishes the proof.
\end{proof}

\begin{dfn}
Let $U$ be a subspace of $\mathcal{H}=\int^\oplus \ell^2 (X(\omega))
d\mu (\omega)$ such that the orthogonal projection $P=P_U$ onto $U$
belongs to $\mathcal{N}$. Then, $\tau (P)$ is called the equivariant
dimension of $U$.
\end{dfn}

For later use we note the following consequence of Theorem
\ref{tracetheorem} and Proposition \ref{Huhu}.
\begin{cor}\label{ArbitraryI}
For each measurable  $I\subset \Gamma$ with $0 < |I|<\infty$ the equation $$\tau (T)
  =\frac{1}{|I|}  \frac{1}{\dens(m)}  \int \Tr (\chi_{I F'} T_\omega) d\mu (\omega)$$
holds for every $T\in
  \mathcal{N}$.
\end{cor}
\begin{proof}
Set $J_0:= B_1 \cap I$ and $J_n :=(B_{n+1}\setminus B_{n}) \cap I$ for
$n\in \NN$. Assume without loss of generality that $|J_n|>0$ for all
$n$. Then, $I$ is the disjoint union of the $J_n$, $n\in \NN_0$ and
each $J_n$ satisfies the assumption of Proposition \ref{Huhu}. Hence,
we obtain
$$\frac{1}{|I|} \int \Tr (\chi_{I F'} T_\omega) d\mu (\omega) =
\frac{1}{|I|} \sum_{n\in \NN} \int \Tr (\chi_{J_n F'} T_\omega) d\mu
(\omega) = \frac{1}{|I|} \sum_{n\in \NN} |J_n| \dens(m)  \tau (T) = \dens(m) \tau (T).$$
\end{proof}

\begin{proof}[Proof of Theorem \ref{main1}]
Set $\nu_H(\varphi) =\tau(\varphi(H))$. Then $\nu_H$  is a positive functional
and hence defines a unique measure. By faithfulness of $\tau$ it is a spectral measure.
By definition, \eqref{e-main1} holds.
\end{proof}

\section{Equivariant dimension of subspaces and jumps of the IDS}
This section is devoted to a proof of  Theorem \ref{main2}. Note that we do
not need ergodicity to derive this result.

\medskip

Recall that every {F\o lner} sequence $(I_n)$ in $\Gamma$ induces a sequence
$\Lambda_n :=I_n F'$ of measurable subsets of $X$. By Proposition \ref{boundary}
$(\Lambda_n)$ is a van Hove sequence in $X$. We start with a technical
result.

Note that for $A_1\subset A_2\subset X$  we can  canonically  regard elements
of $ \ell^2 (A_1(\omega))$ as elements of $\ell^2 (A_2(\omega))$, as well,
by extending them by zero.

\begin{lem}
  Let $P\in \mathcal{N}$ with $P\geq 0$ and $\tau (P) >0$ be given.  Let $R>0$
  and a {F\o lner} sequence $(I_n)$ in $\Gamma$ be given. Then, there exists an
  $N\in \NN$ and a set $\widetilde{\Omega}$ in $\Omega$ of positive measure
  such that for all $\omega \in \widetilde\Omega$
\[
\ran (p_{\Lambda_N (\omega)} P_\omega) \cap \ell^2 (\Lambda_{N,R}
(\omega)) \neq \{0\}.
 \]
\end{lem}
\begin{proof} Without loss of generality we can assume that the density $\dens(m)$ equals $1$.
  As $P$ be belongs to $\mathcal{N}$, the function $\omega \mapsto
  \|P_\omega\|$ is essentially bounded. We can assume without loss of
  generality that this constant is equal to $1$.  We set $\delta
  :=\tau (P) >0$.
By Proposition \ref{boundary}, there exists $N\in \NN$ with
  $$\omega(\partial^R \Lambda_N) \leq \frac{\delta}{2} |I_N|$$
  for   all $\omega \in \Omega$.  As by Corollary
  \ref{ArbitraryI},
$$\delta = \tau (P) = \int \frac{1}{|I_N|} \Tr (\chi_{\Lambda_N} P_\omega) d\mu
(\omega),$$
there exists a set $\widetilde{\Omega}$ of positive measure with
\[
\frac{1}{|I_N|} \Tr (\chi_{\Lambda_N} P_\omega) \geq \delta
\]
for all $\omega \in \widetilde{\Omega}$. This gives
$$
\dim ( \ran  (p_{\Lambda_N (\omega)} P_\omega )) \geq \Tr
(p_{\Lambda_N (\omega)} P i_{\Lambda_N (\omega) })\geq \delta |I_N|$$
for all $\omega\in\widetilde{\Omega}$.  The statement follows from Lemma
\ref{LA} with $U = \ran  (p_{\Lambda_N (\omega)} P_\omega )$, since then
$U_R = \ran (p_{\Lambda_N (\omega)} P_\omega) \cap \ell^2 (\Lambda_{N,R} (\omega))$.
\end{proof}

\begin{lem} Let $H$ satisfy (A). Let $\lambda\in \RR$ be arbitrary.  Set
  $P:=E_H (\{\lambda\})$ and denote by $ P_{comp} $  the projection on the closure of the
  linear hull of compactly supported eigenfunctions to $\lambda$. Then, the
  following holds:
\begin{itemize}
\item[(a)]  $P = P_{comp}$.
\item[(b)] If $P_{comp}\neq 0$, then $\tau (P) >0$.
\end{itemize}
\end{lem}
\begin{proof}
  (a) If $P=0$ the statement is clear. We therefore only consider the
  case $P\neq 0$, i.e. $\tau (P)>0$.  Assume $P\neq P_{comp}$. Then,
  $Q:=P - P_{comp}$ is a projection with $Q\neq 0$.  Hence $\tau (Q)
  >0$ as $\tau$ is faithful.  Set $R$ to be twice the hopping range of
  $H$. By the previous
  lemma, there exist $N\in \NN$ and a set $\widetilde{\Omega}$ of
  positive measure in $\Omega$ such that for each $\omega\in \widetilde{\Omega}$
$$\mbox{ran} (p_{\Lambda_{N} (\omega)} Q_\omega) \cap \ell^2
(\Lambda_{N,R} (\omega)\neq \{0\}.$$
By definition of the hopping
range this gives compactly supported eigenfunctions in the range of
$Q_\omega$ for all $\omega\in \widetilde{\Omega}$. This is a
contradiction.

 (b) This follows as $\tau$ is faithful.
\end{proof}

\begin{proof}[Proof of Theorem \ref{main2}]
This is a direct consequence of the previous lemma.
\end{proof}

\begin{proof}[Proof of Corollary \ref{Korollar}]
As in Theorem \ref{main2}, let $P_{comp,\omega}$ be the projection onto the space of compactly supported solutions of $(H_\omega -\lambda)u=0$ and $P_{comp}=(P_{comp,\omega})$.
Let $\Omega'$ be the set of $\omega\in \Omega $ with $P_{comp,\omega}\neq 0$.
Then, $\Omega'$ is invariant and measurable, hence by ergodicity it has measure $0$ or $1$.

Now,  by Theorem \ref{main2} and the faithfulness of $\tau$, the function  $N_H$ is discontinuous at $\lambda$ (i.e. (i) holds)  if and only if $P_{comp}\neq 0$.  By definition,  $P_{comp} \neq 0$ if and only if $\Omega'$ has positive measure (i.e. (ii) holds). As  $\Omega'$ has either measure $0$ or meassure $1$, it has positive measure if and only if it has measure $1$ (i.e. (iii) holds).
\end{proof}

\section{Some deterministic approximation  results}
In this section we present three deterministic results. The first two results give estimates for the difference
\[
 \chi_\Lambda \phi (H_\omega) - \phi (p_{\Lambda (\omega)}    H_\omega  \, i_{\Lambda (\omega)} )
\]
for suitable functions $\phi$ on $\RR$ and $\Lambda \subset X$.  In one
way or other this type of result is entering all considerations on
convergence of the IDS (see the first list of references in the Introduction).
On the technical level our arguments are related to \cite{Eckmann-99,Elek-03,MathaiSY-03}.

The last result gives a lemma from measure theory on convergence of
distribution functions for measures on $\RR$. This type of lemma seems
to have first been used in the present context in \cite{Elek-b}.  There,
one can also find a proof. For the convenience of the reader we
include an alternative proof (and actually give a slightly
strengthened result).

\begin{lem}\label{vague}
Let $(I_n)$ be a {F\o lner} sequence in $\Gamma$ and $\Lambda_n:= I_n F'$ as before.
Let $H$ satisfy (A). Let $\phi$ be a continuous function on $\RR$.  Then,  for any $\omega\in \mathcal{D}$
\[
 \frac{|\Tr (\chi_{\Lambda_n} \phi (H_\omega))-\nu_{\omega,n} (\phi)   |   }{|I_n|} \longrightarrow 0
\]
for $n\to \infty$.
\end{lem}
\begin{proof}
First we prove that the statement holds if $\phi(x) = x^k$ for some $k \in \NN$.
Note that $\nu_{\omega,n} (\phi) = \Tr(H_{\omega,n}^k)$ is a sum over closed paths of lenght $k$,
more precisely
\[
 \sum _{x \in \Lambda_n(\omega)} \
\sum_{x_1, \dots, x_{k+1} \in \Lambda_n(\omega)} H_{\omega,n}(x_1, x_2) \dots H_{\omega,n}(x_k, x_{k+1})
\]
where in the second sum $x_1 =x= x_{k+1}$. Since $\Tr (\chi_{\Lambda_n} H_\omega^k) $ can be written
in  a similar way, the difference $\Tr (\chi_{\Lambda_n} H_\omega^k)  -\Tr(H_{\omega,n}^k)$ is bounded in modulus by
\begin{multline*}
\sum _{x \in \partial^{k R} \Lambda_n(\omega)}  \
\sum_{x_1, \dots, x_{k+1} \in X(\omega)} |H_{\omega,n}(x_1, x_2)| \dots |H_{\omega,n}(x_k, x_{k+1})|
\\
\le
\sum _{x \in \partial^{k R} \Lambda_n(\omega)} \omega(B_{kR})^k \|H\|^k
\le
\omega(\partial^{k R} \Lambda_n) (M_{kR})^k \|H\|^k.
\end{multline*}
Here $R$ denotes the finite hopping range of the operator $H_\omega$.
Now one uses Proposition \ref{boundary} and the fact that $I_n, n \in \NN$ is a {F\o lner} sequence
to conclude that
\[
 \lim_{n\to\infty}\frac{\omega(\partial^{k R} \Lambda_n) }{|I_n|} (M_{kR})^k \|H\|^k=0
\]
Now the following considerations extend the convergence result to all $\phi \in C(\RR)$:

Note that the only relevant data of the function $\phi$ are
  its values on the spectrum of $H_\omega$, which is a compact set as
  $H_\omega$ is bounded. Let $\mathcal{S}$ be the family of functions for
  which the statement of the lemma holds. We just proved that all polynomials belong to $\mathcal{S}$.
  Moreover, the set $\mathcal{S}$ is obviously closed under uniform
  convergence. The statement follows from Stone-Weierstrass theorem.
\end{proof}

\begin{lem}\label{discontinuity}
Let $(I_n)$ be a {F\o lner} sequence in $\Gamma$ and $\lambda\in \RR$ arbitrary.
Then, for any $\omega\in \tcD$
\[
 \frac{|\Tr(\chi_{\Lambda_n}E_\omega(\{\lambda\})) -\nu_{\omega,n} (\{\lambda\})  |   }{|I_n|} \longrightarrow 0
 \]
for $n\to \infty$.
\end{lem}
\begin{proof} Let $R$ be the hopping range of $H$.  Fix $\omega\in \mathcal{D}$. Let $V_n$ be the subspace of all
  solutions of $(H_\omega - \lambda) v =0$, which vanish outside
  $\Lambda_{n,R}$. Let $D_n$ be the dimension of $V_n$.

We now apply  Lemma \ref{LA} to the space $U$ of all solutions of $(p_{\Lambda_n (\omega)} H_\omega i_{\Lambda_n (\omega)} - \lambda ) u =0$   and note that $U_R = V_n$ as the hopping range of $H_\omega$ is $R$. This gives
$$ 0 \le \nu_{\omega,n} (\{\lambda\}) - D_n \leq \omega(\partial^R (\Lambda_n)).$$
Moreover, obviously,
$$ D_n \leq \Tr (\chi_{\Lambda_n} E_\omega (\{\lambda\}) ) \leq
\dim(\ran (\chi_{\Lambda_n} E_\omega (\{\lambda\}) )).$$
We now apply  Lemma \ref{LA} to $U':= \ran \chi_{\Lambda_n} E_\omega (\{\lambda\})$ and note that $U'_R = V_n$ as the hopping range of $H_\omega$ is $R$. This gives
$$ 0 \le \Tr (\chi_{\Lambda_n} E_\omega (\{\lambda\})) - D_n \leq \omega(\partial^R (\Lambda_n)).$$
By Proposition \ref{comparable}, there exists $\rho>0$ with  $\partial^R
(\Lambda_n) \subset (\partial^\rho I_n) F'$ for all $n\in \NN$. The statement
follows now from the triangle inequality and Proposition \ref{boundary}.
\end{proof}

\begin{lem} \label{vagueandmore}
Let $\nu$ be a probability measure on
$\RR$.  Let $(\nu_n)$ be a sequence of bounded  measures on
$\RR$ which  satisfy
\begin{itemize}
\item $\nu_n$ converge weakly to the measure $\nu$,
\item $\nu_n (\{\lambda\}) \longrightarrow \nu (\{\lambda\})$ for all $ \lambda\in \RR$.
\end{itemize}
Then, the distribution functions $\lambda\mapsto \nu_n
((-\infty,\lambda])$ of the $\nu_n$ converge with respect to the
supremum norm to the distribution function $\lambda\mapsto \nu(
(-\infty,\lambda])$ of $\nu$.
\end{lem}
\begin{rem}
Note that vague convergence $\nu_n \to \nu$,  actually implies weak convergence,
if one of the two following conditions hold:
\begin{itemize}
 \item all $\nu_n$ are probability measures, or
 \item there exists a compact interval such that the supports of $\nu$ and of $\nu_n$ is contained
in this interval for all $n\in \NN$.
\end{itemize}

\end{rem}

\begin{proof}
Let $\epsilon >0$ be arbitrary.
Let $\nu_c$ denote the continuous part of $\nu$ and $\nu_p$ the point
part of $\nu$.  Choose
\[
-\infty = t_{-1}= t_0< t_1< \ldots < t_L < t_{L+1}  <  t_{L+2} = \infty
\]
such that
\begin{itemize}
\item[($*$)] $\nu_c ((t_j,t_{j+1}])\leq \epsilon$ for $j=1,\ldots L$,
$\nu (-\infty,t_1) \leq \epsilon$ and $\nu (t_{L+1},\infty) \leq \epsilon$.
\end{itemize}
Such a choice is possible since $\nu$ is a probability measure.
Let $(\lambda_k)$
be an enumeration of the points of discontinuity of $\nu$. Assume
without loss of generality that $k$ runs through all of $\NN$.  Choose
$N\in \NN$ with
\begin{itemize}
\item[($**$)]$\sum_{k=N+1}^\infty \nu (\{\lambda_k\})\leq
\epsilon$.
\end{itemize}
Choose continuous functions $\phi_j$
with  $\chi_{(-\infty,t_j]}\leq \phi_j\leq \chi_{(-\infty,
t_{j+1}]}$ for $j=0,\ldots, L$.  Set $\phi_{-1} \equiv 0$.
For all large enough $n$ we then have
\begin{itemize}
\item[($***$)] $|\nu_n (\phi_j) - \nu (\phi_j)|\leq \epsilon $ for $j=1,\ldots, L+1$,
\item[($****$)] $|\nu_n ( \chi_{\{\lambda_j\mid j=1,\ldots, N\}}) - \nu(\chi_{\{\lambda_j\mid j=1,\ldots, N\}} ) |\leq \epsilon$.
\end{itemize}
For such $n$ we prove now
\[
\nu ((-\infty,\lambda]) - \nu_n((-\infty,\lambda])\le 5 \epsilon
\]
for all $\lambda\in\RR$ as follows:
Choose $j\in0,\ldots, L+1$ with $t_{j}\leq \lambda < t_{j+1}$
and define $\psi$ by
$$\chi_{(-\infty,\lambda]} = \phi_{j-1} + \psi$$ be given.
Since $0\leq \psi \leq 1$ on $\RR$ and $\supp \psi \subset [t_{j-1},t_{j+1}]$ we then obtain
\begin{eqnarray*}
\nu_n ((-\infty,\lambda]) &= & \nu_n (\phi_{j-1}) + \nu_n
  (\psi)\\
(***)\;&\geq &\nu (\phi_{j-1}) - \epsilon + \nu_n (\psi)\\
&= & \nu((-\infty,\lambda])- \nu(\psi) - \epsilon + \nu_n (\psi)\\
(*)\;&\geq &  \nu((-\infty,\lambda]) - \nu_p (\psi) - 3 \epsilon + \nu_n
(\psi)\\
&=& \nu((-\infty,\lambda])- 3 \epsilon + \nu_n (\psi) - \nu_p (\psi)\\
(**)\:&\geq & \nu((-\infty,\lambda])- 4 \epsilon +   \nu_n
(\psi) -  \nu_p (\psi \chi_{\{\lambda_j\mid j=1,\ldots, N\}})\\
&\geq & \nu((-\infty,\lambda])- 4 \epsilon  + \nu_n
(\psi     \chi_{\{\lambda_j\mid j=1,\ldots, N\}} ) -  \nu_p(\psi \chi_{\{\lambda_j\mid j=1,\ldots, N\}})\\
(****) &\geq &  \nu((-\infty,\lambda]) - 5 \epsilon.
\end{eqnarray*}
Note that the above inequalities hold also if we replace $\nu((-\infty,\lambda]), \nu_n ((-\infty,\lambda]) $
by $\nu((-\infty,\lambda)), \nu_n ((-\infty,\lambda))$.
Thus we have proven
\[
\lim_{n\to\infty}\sup_{\lambda \in \RR} \nu ((-\infty,\lambda]) - \nu_n((-\infty,\lambda])\le 0
\]
To prove the opposite inequality we use the measure $\tilde \nu$ which is the reflection of $\nu$ around the origin.
It inherits all properties of $\nu$ which have been used in the previous calculation.
Then
\[
\tilde \nu ((-\infty, \lambda)) = \nu([-\lambda, \infty))=1- \nu((-\infty, -\lambda])
\]
The above argument yields
$\lim_{n\to\infty}\sup_{\lambda \in \RR} \tilde\nu ((-\infty,\lambda)) - \tilde\nu_n((-\infty,\lambda))\le 0$.
Since $\tilde\nu ((-\infty,\lambda)) - \tilde\nu_n((-\infty,\lambda))= \nu_n((-\infty, -\lambda])-\nu((-\infty, -\lambda])$,
the proof is completed.
\end{proof}

\section{Proof of Theorem \ref{main3}}
In this section, we prove Theorem \ref{main3}. Since here ergodicity plays a role, we
first discuss various consequences of ergodic theorems.

\medskip

For a function  $v$ on $\Omega$ and $I\subset \Gamma$
with $|I|>0$ we define the function  $v_I$ on $\Omega$ by
$$ v_I (\omega) :=\int_{I} v (\gamma^{-1}\omega) d\gamma.$$
For any $v$ integrable with respect to $\mu$ and any tempered {F\o lner} sequence $(I_n)$ there exists
a $\Gamma$-invariant $\bar v\in L^1(\mu)$ such that
\[
\lim_{n\to \infty} \frac{1}{|I_n|} v_{I_n} (\omega) = \bar v(\omega)
\]
for almost every $\omega\in \Omega$ and in $L^1$, see \cite{Lindenstrauss-01}.
Moreover, if $\mu$ is ergodic $\bar v (\omega)= \int v(\omega) d\mu(\omega)$
almost surely.

Our natural setting is not concerned with $v_I$ but rather with functions of
the form $\chi_{ I F'}$. We next show that these two types of functions are
comparable. To do so, we recall that we have fixed a point $p\in F'\subset X$.

\begin{prp} \label{uvsF}
Let $u$ be a measurable nonnegative function on $X$ satisfying
\eqref{e-admissible}.  Let $r>0$ such that the support of $u$ is
contained in the open ball $B_r(p)$ around $p$ with radius $r$. Then,
$$ | \chi_{ I F'}(x)  - u_I (x) |\leq
\chi_{\partial^r (I F')} (x)$$
for any $x\in X$.
\end{prp}
\begin{proof}
Consider first $x\in (I F')_r$: Then, $B_r (x) \subset I F'$. Any
$\gamma\in \Gamma$ with $u(\gamma^{-1} x) \neq 0$ satisfies
$d(\gamma^{-1} x, p) < r$, hence
$$ \gamma p  \in B_r (x) \subset I F'$$
and thus $\gamma \in I$.  For such $x$ we therefore obtain
$$ \int_I u (\gamma^{-1} x) d\gamma = \int u (\gamma^{-1} x) d\gamma =1.$$

Consider now $x\notin (I F')^r$: Then, $B_r (x) \subset X\setminus I
F'$.  Any $\gamma\in \Gamma$ with $u(\gamma^{-1} x) \neq 0$ satisfies
\[
\gamma p \in B_r (x) \subset X\setminus (I F')
\]
 and hence $\gamma \notin I$.  For such $x$ we therefore obtain
$$  \int_I u (\gamma^{-1} x) d\gamma =0.$$

Consider now $x\in (I F')^r \setminus (I F')_r$: As $u$ is nonnegative
with $\int u (\gamma^{-1} x) d\gamma =1$, we obtain $ 0 \leq u(x) \leq
1$ for such $u$.

Having considered these three cases, we can easily
obtain the statement.
\end{proof}

We now derive two consequences of the previous proposition and the ergodic
theorem.

\begin{prp}\label{erg}
  Let $(I_n)$  a tempered {F\o lner} sequence in $\Gamma$ and $T\in \mathcal{N}$ be
  arbitrary. Then, for almost every $\omega\in \Omega$,
\[
 \lim_{n\to\infty} \frac{1}{|I_n|} \Tr (\chi_{\Lambda_n}  T_\omega) = g(\omega).
\]
where $g \in L^1(\mu)$ is $\Gamma$-invariant. The convergence holds also in $L^1$ sense.
If $\mu$ is ergodic,
$g=\dens(m) \tau (T)$ almost surely.
\end{prp}
\begin{proof}

We show that $\omega\mapsto  \frac{1}{|I_n|} \Tr (\chi_{\Lambda_n}  T_\omega)$
converges almost surely and in $L^1$ for $n\to \infty$.
By Proposition \ref{boundary} combined with Proposition \ref{uvsF}, it suffices to consider the sequence
\[
 \Tr (u_{I_n}  T_\omega)  = \sum_{x\in X(\omega)}   u_{I_n}(x) T_\omega (x,x)
\]
with $u$ satisfying \eqref{e-admissible}, instead of considering
\[
\Tr (\chi_{\Lambda_n}  T_\omega)  =
\sum_{x\in \Lambda_n(\omega)}\ T_\omega(x,x)
\]
Define with such a $u$
$$ v(\omega) := \Tr (u \, T_\omega) = \sum_{x\in X(\omega)} u(x) T_\omega (x,x).$$
By the ergodic theorem, $  v_{I_n} /|I_n| $ converges  almost surely
to some $\Gamma$-invariant $g\in L^1(\mu)$. A direct calculation using equivariance shows
\[
v_{I_n} (\omega) 
= \sum_{x\in X(\omega)}   u_{I_n}(x) T_\omega (x,x) .
\]
Thus the first statement of the Proposition is proven.
If $\mu $ is ergodic,  then $g
=\int \Tr (u \, T_\omega) d\mu(\omega)=\dens(m) \tau (T)$ almost surely.
\end{proof}

We note the following special case of the previous proposition.

\begin{cor}\label{Ivso}
  Let $\mu$ be ergodic and $(I_n)$ be a tempered {F\o lner} sequence in $\Gamma$.  For almost every $\omega\in
  \Omega$, we have $\lim_{n\to\infty} \frac{\omega (\Lambda_n)}{|I_n|} =
  \dens(m)=\dens (m) \tau (\Id)$. In particular, $\tau (\Id)=1$.
\end{cor}
\begin{proof} The previous proposition with $T=\Id$ shows
  pointwise and $L^1$ convergence of the functions $\omega \mapsto
  \frac{\omega (\Lambda_n)}{|I_n|}$ to the constant $\dens(m)\tau(\Id)$. Since
\[
\text{dens}(m) = \frac{m(\Lambda_n)}{|I_n|} =
\int \frac{\omega(\Lambda_n)}{|I_n|} d\mu(\omega)
\]
 for all  $n\in\NN$, $\tau(\Id)$ must be equal to  $1$  and the statement follows.
\end{proof}

\begin{proof}[Proof of Theorem \ref{main3}]
Since $\frac{N_{\omega,n}}{\omega(\Lambda_n)}= \frac{N_{\omega,n}}{|I_n|}\frac{|I_n|}{\omega(\Lambda_n)}$
and $\frac{|I_n|}{\omega(\Lambda_n)}$ converges to $\dens(m)^{-1}$ almost surely,
it suffices to show convergence of $\frac{N_{\omega,n}}{|I_n|}$.
 There are at
  most countably many points of discontinuity of $\nu_H$. Thus, Lemma
  \ref{discontinuity} combined with Proposition \ref{erg}, gives convergence
  in all points of discontinuity of $\nu_H$ almost surely. On the other hand, note that
  the space of continuous functions with compact support on $\RR$ is
  separable. Thus, weak convergence of probability measures follows from
  convergence on a countable dense subset of  continuous functions on $\RR$
  with compact support.
Thus, Lemma
  \ref{vague} combined with Proposition \ref{erg} gives weak convergence of
  the measures almost surely. Now, the result follows from Lemma \ref{vagueandmore}.
\end{proof}

\begin{Rem} \label{r-non-ergodic}
Even in the case that $\mu$ is not ergodic we can prove a uniform convergence statement.
By assumption (A)  there is an $R'\in \RR $ such that $\sigma(H_\omega) \subset [-R',R']$.
Let $\psi_j, j\in\NN$ be a countable dense set in $C([-R',R'])$.
By Proposition \ref{erg} there exists  a set of full measure $\Omega_j\subset\Omega$
such that for all $\omega \in\Omega_j$
\[
  l_\omega(\psi_j):=\lim_{n\to\infty} \frac{1}{|I_n|} \Tr (\chi_{\Lambda_n}  \psi_j(H_\omega))
\]
exists. This way one defines for all $\omega \in \tilde \Omega:=\cap_{j\in\NN}\,  \Omega_j$ a positive, bounded linear functional
$l_\omega$, i.e.~a measure. In particular one may define a density for such $\omega$ by
$\dens(\omega) =\lim_{n\to\infty} \frac{\omega(\Lambda_n)}{|I_n|}$.
For $\omega \in \tilde \Omega$ with $\dens(\omega)>0$ we can as before renormalize the sequence
$\frac{N_{\omega,n}}{\omega(\Lambda_n)}= \frac{N_{\omega,n}}{|I_n|}\frac{|I_n|}{\omega(\Lambda_n)}$
which thus converges to the distribution function of $\nu_\omega:=\frac{l_\omega}{\dens \omega }$.
If  $\dens(\omega)=0$ we still see that
$\Tr (\chi_{\Lambda_n}  \psi(H_\omega)) \le \|\psi\|_\infty \omega(\Lambda_n)$. Thus if $X(\omega)$
is not empty
\[
\frac{\Tr (\chi_{\Lambda_n}  \psi(H_\omega))}{\omega(\Lambda_n)}
\le \|\psi\|_\infty
\]
and we can conlude by the Banach-Alaouglu theorem that there is a subsequence along which
$\displaystyle \frac{N_\omega^{n_k}}{\omega(\Lambda_{n_k})}$ converges weakly.
A posteriori, we can enhance this to uniform convergence using Lemma \ref{vagueandmore}.

\end{Rem}

\section{Models with aperiodic order}\label{s-Aperiodic}
In this section, we discuss models with aperiodic order. In these
cases $\Gamma = X=\RR^m$ is a continuous group.  We recover the main
result of \cite{KlassertLS-03} concerning characterization of jumps of the IDS
via compactly supported eigenfunctions. In fact, we obtain a
strengthening of the result of \cite{KlassertLS-03} in three respects: We do not
need an ergodicity assumption anymore, we do not need a finite local complexity assumption  and we identify the size of the
jumps as a equivariant dimension. (Note that the latter, however, can
directly derived from the convergence statement in \cite{LenzS-06}).  We
also obtain a result on convergence of the IDS. This result, however,
is strictly weaker than the results of \cite{LenzS-06} as it neither holds
for all $\omega\in\Omega$ nor gives an explicit error bound on the speed
of convergence.

\smallskip

The setting is as follows: There is an obvious action of
$\Gamma=\RR^m$ on $X=\RR^m$ by translation.  A fundamental domain is
given by the compact set $\{0\}$ and the map $\Phi \colon \Gamma
\longrightarrow \RR^m$ is just the identity and therefore an isometry.
Hence, the geometric assumptions of our setting are satisfied.
The ball around $x\in \RR^m$ with radius $r$ is denoted by $B_r (x)$.
As before, $\mathcal{D}$ is the
set of subsets of $\RR^m$ whose elements have  Euclidian
distance at least $1$.
We call a subset $\cM$ of $\cD$ a collection of Delone sets
if the following holds:
\begin{itemize}
\item[--] There exists an $R'>0$ with $A \cap B_{R'} (x)\neq \emptyset$ for every $A\in\cM$ and $x\in X$.
\end{itemize}
It is said to be of finite local complexity if it also satisfies the following:
\begin{itemize}
\item[--] For each $r>0$, the set $\{ (A-x)\cap B_r(0)\mid A \in\cM, x\in A\}$ is finite.
 \end{itemize}

Let $\mu=\mu_{\cM}$ be an invariant probability measure
on $\mathcal{D}$ whose support $\Omega$ is a collection of Delone sets.
In particular $X(\omega) = \omega \in \Omega$ is a discrete subset of $\RR^d$. Assume that the density of $m$ is $1$.

This setting gives a notion of an equivariant operator as a family $(H_\omega)$
of operators $H_\omega\colon\ell^2 (X(\omega))\longrightarrow \ell^2 (X(\omega))$
with
$$ H_{\gamma + \omega} ( \gamma +x,\gamma +y) = H_\omega (x,y).$$
The natural trace is defined via
$$\tau (H) = \int_\Omega \Tr (u H_\omega) d\mu (\omega),$$
where the continuous  $u\colon  \RR^m\longrightarrow \RR$ is an arbitrary function   with compact support and  $\int_\RR^m  u (x) dx =1$.

\smallskip

For an operator satisfying (A) and $\lambda\in\RR$ our abstract results give:

\begin{itemize}
\item[(i)]  Let $U_\omega$ be the subspace of $\ell^2(X(\omega))$ spanned by compactly supported solutions to  $(H_\omega-\lambda)u =0$.
Then, the value $\nu_H(\{\lambda\}) =N_H (\lambda) - N_H (\lambda-) = \tau (E_H (\{\lambda\})$ is the
equivariant dimension of the subspace $\int^\oplus U_\omega d\mu(\omega)$ of $\int^\oplus_\Omega \ell^2 (X(\omega))d\mu (\omega)$.
\item[(ii)] If $\mu$ is furthermore assumed ergodic, then  $\lambda$ is a point of discontinuity of $N_H$
if and only if there exists a compactly supported  eigenfunction of $H_\omega$ to $\lambda$  for almost every $\omega\in \Omega$.  In this case, these compactly supported eigenfunctions span the eigenspace of $H_\omega$ to $\lambda$.
\item[(iii)] The normalized finite volume counting functions
$\frac{N_{\omega,n}}{|\Lambda_n|}$ converge almost surely with respect to the supremum norm
towards the function $N_H$, i.e.
\[
\lim_{n\to\infty}\Big\|\frac{1 }{|\Lambda_n|}N_{\omega,n}- N_H \Big\|_\infty =0.
\]
\end{itemize}

\section{Periodic models on amenable graphs and CW-complexes}
\label{PeriodicExamples}
In this section, we briefly discuss how the corresponding results of
\cite{DodziukLMSY-03,DodziukM-98,Eckmann-99,Kuchment-05,MathaiSY-03,MathaiY-02} can be recovered in our framework.  In
these cases the group $\Gamma$ as well as the space $X$ are discrete and countable.

\smallskip

The geometric setting in the cited works is given either by a graph
or a CW-complex. Let us first consider the case that a graph $G=(V,E)$ with
vertex set  $V$ and edge set $E$ and a finitely generated amenable group
$\Gamma$ are given, such that each vertex degree is finite and $\Gamma$ acts
freely and cocompactly on $G$ by automorphisms.  We set $X:=V$ and show that the
assumptions of our setting are satisfied: A finite set $S_0$ of
generators of $\Gamma$ which is symmetric in the sense that
$S_0=S_0^{-1}:= \{\gamma^{-1}\mid \gamma \in S_0\}$ defines a word metric
$d_{S_0}$ on $\Gamma$. Obviously, $d_{S_0}$ is invariant under the action of $\Gamma$
on itself. If we choose a different (finite, symmetric) set of generators $S$
it gives rise to another metric $d_S$. This metric is equivalent to the metric $d_{S_0}$ and in particular  the two spaces
$(\Gamma, d_S)$ and $(\Gamma, d_{S_0})$ are roughly isometric.

As $\Gamma$ acts cocompactly, there exists a compact (i.e.~finite)
fundamental domain $F'$ of the action of $\Gamma$. In particular, $F'$
is equal to its closure $F$. As $\Gamma$ acts freely, we obtain a well
defined map $$\Phi \colon X \longrightarrow \Gamma, \mbox{with} \; \: x\in
\Phi(x) F.$$ We have to show  hat this map is a rough isometry. To do so,
we need of course a metric on $X$. Let us first assume that the graph is connected.
Then $X=V$ carries a
natural metric $d$ coming from finite paths between the points. This
metric is obviously $\Gamma$-invariant. Set
$$ C:=\max_{s\in S\cup\{\id\}} \{ d(x,y)\mid x\in F, y\in s F \} < \infty. $$
Then,
$$ d(x,y) \leq C d_\Gamma (\Phi(x),\Phi (y)) + C.$$ As for the
converse inequality, we need some more preparation. We say that
$\gamma\in \Gamma$ is a neighbor of $\rho\in\Gamma$ if there exists
an edge connecting a vertex in  $\rho F$ with a vertex in $\gamma F$.
Denote the  set of neighbors of $\id\in\Gamma$ by $S_0$.
Since $X$ is connected, $S_0$ is a set of generators of $\Gamma$;
since each vertex degree is finite, $S_0$ is finite; and by the properties of the action of $\Gamma$,
$S_0$ is symmetric.
Thus
\[
d_{S_0} (\Phi (x), \Phi (y)) \leq d(x,y).
\]
Since any word metric on $\Gamma$ is roughly isometric with
$d_{S_0}$,  this shows that $\Phi$
is indeed a rough isometry.

If the graph is not connected, there is no natural choice of a metric on $X$.
In this case, we can induce a metric on $X$ by the metric on $\Gamma$ and $\Phi$
in two steps: the metric on the group $\Gamma$ defines a distance between different fundamental domains.
We can assume that this distance function takes values in $\NN$.
Within the fundamental domains the distance between two points is defined using shortest paths.
Let us scale the latter distance function such that the diameter of a fundamental domain is bounded by one.
This way one obtains a metric on $X$ which is by construction roughly isometric to $\Gamma$.
\medskip

Let us now turn to the case of CW-complexes.
Thus let a CW-complex $Y$ and a finitely generated amenable group $\Gamma$
be given which acts freely on $Y$ by automorphisms. We assume that the quotient $Y/\Gamma$
is a CW-complex of finite type, i.e.~all its skeleta are finite.
For a $j \in \NN$ denote by $Y_j$ the set of $j$-cells in
$Y$. Two such cells are called adjacent if either the intersection of their closures
contains a $j-1$ cell of $Y$, or if both are contained in a the closure of a single $j+1$ cell.
Since we assumed that the quotient $Y/\Gamma$ is of finite type, the number of cells adjacent to any given cell is finite.
Now we fix $j \in \NN$  and define a graph $G_j$ with vertex set $V=Y_j$.
Two elements $V$ are connected by an edge iff they are adjacent.
Each automorphism of the original CW-complex induces a graph-automorphism on $G_j$.
In particular, $\Gamma$ acts freely and cocompactly on $G_j$.
Thus we are back in the setting which we discussed at the start of this section.
(Note that for each $j \in \NN$ we extract
from the CW-complex $Y$ a different graph $G_j$ and correspondingly the graph Hamiltonians,
which we define below, will also depend on $j$.)

These considerations show that the geometric assumptions of our model are
satisfied in the cited works.
\smallskip

In the present setting the set $X$ itself belongs to $\mathcal{D}$ and is in fact
invariant under the action of $\Gamma$.  Thus, we can choose the
measure $\mu$ on $\mathcal{D}$ to be supported on $\{X\}$. This means that
$\Omega$ consists of a single element which is just $X$. Thus,
everything depending on the family $\omega\in \Omega$ is replaced by a
single object in the sequel. In particular, we certainly have all
ergodic assumptions satisfied. In fact, all statements concerning
almost sure convergence in $\Omega$ give deterministic statements.
\smallskip

We will now deal with the operator theoretical side of things.  In
\cite{MathaiY-02,MathaiSY-03} one is given a family $s_\gamma$,$\gamma\in\Gamma$ of
maps $s_\gamma \colon V\longrightarrow \{z\in \CC\mid |z|=1\}$. In the other
cases one just sets $s_\gamma \equiv 1$,$\gamma\in\Gamma$.  The family
of operators in question is then given by a single operator $H=H_X$
satisfying
$${s_\gamma (x)} H(\gamma x,\gamma y) \overline{s_\gamma (y)} = H(x,y).$$
The natural trace becomes
$$\tau (H) = \Tr (\chi_F H).$$

As before denote by $H_n$ the restriction of the operator
$H$ to $\Lambda_n= I_n F$, where $I_n\in\Gamma$ is a {F\o lner} sequence.
With the usual convention $N_H (\lambda):=\tau ( E_H (\lambda) )$ and $N_{H,n}:=\Tr E_{H_n}(\lambda)$, our results can the be re formulated as follows:

\begin{itemize}
\item[(i)] For any $\lambda\in \RR$, the value $\nu_H(\{\lambda\})=  \tau (E_H (\{\lambda\})
$ is
the equivariant dimension of the subspace of $\ell^2 (X)$ spanned by compactly supported solutions of $(H_\omega -\lambda)u =0$.
\item[(ii)]  A number $\lambda\in \RR$ is a point of discontinuity of $N_H$ if and only if there exists a compactly supported eigenfunction of $H$ to $\lambda$.
\item[(iii)] The functions $N_{H,n}$ converge with respect to the supremum norm
towards the function $N_H$, i.e.
\[
\lim_{n\to\infty}\Big\|\frac{1 }{|\Lambda_n|} N_{H,n}- N_H \Big\|_\infty =0.
\]
\end{itemize}

\section{Anderson and percolation Hamiltonians}\label{s-APH}

In this section we discuss the application of our results to certain types of random models on graphs.
More precisely, we consider Anderson models, random hopping models, as well as site and bond percolation models.
They can be understood as randomized versions of the operators introduced in Section \ref{PeriodicExamples}.
In particular we are again given a graph $G=(V,E)$ with bounded vertex degree on which a finitely generated,
amenable group $\Gamma$ acts freely and cocompactly by automorphisms.

The application of our abstract theorems to these models recover and extend in particular the results
which concern the construction of the IDS by its finite volume analogues
obtained in \cite{Veselic-05a,KirschM-06,LenzMV}.

%
%
%
%

Let us introduce the \emph{Anderson-Percolation Hamiltonian}. 
We set $X=V$ and assume that a function $h_0\colon X\times X \to \CC$ is given with
$h_0(x,y) =  \overline{h_0(y,x)}$ and $h_0(x,y)\neq 0\Rightarrow d(x,y) <R$.
Consider the setting from Section \ref{s-Setting} and assume additionally
that there exist a constant $C$ and a function
$\cV_\omega\colon X \times X\to [-C,C]$ with support on the diagonal  $D:=\{(x,x)\mid x \in X\}$ sucht that
for all $\omega =(A,h)\in \Omega$
\[
 h = \big ( h_0 + \cV_\omega \big ) \chi_{A\times A}.
\]
The operator associated  to $\omega\in \Omega$ acts on the $\ell^2$ space of the diluted graph $X(\omega)=A$.
More precisely for each $v \in \ell^2(X(\omega))$ and $x \in X(\omega)$
\[
 (H_\omega v)(x)= \Big (\sum_{y\in X(\omega)} h_0(x,y) v(y) \Big ) + \cV_\omega v(x).
\]
We can think of the first term as the kinetic energy or hopping term and of the second as a random potential.
In the special case that $X(\omega) = X$ for all $\omega \in \Omega$ and
\[
h_0(x,y) = \chi_{1}(d(x,y)):= \begin{cases}
		             1 &\text{ if }  d(x,y) =1, \\
           		     0 & \text{ otherwise}
           		     \end{cases}
\]
we obtain the Anderson model.
Likewise, for $\cV_\omega \equiv 0$ and $h_0(x,y) = \chi_{1}(d(x,y))$ we have the site-percolation Hamiltonian.
Of course it is possible to choose the measure $\mu$ in such a way that the discussed models are i.i.d.
with respect to the coordinates $x\in X$, see cf.~\cite{Veselic-05b}.

Let us now discuss the random hopping model, which includes the bond percolation model as a special case.
Such models have bee considered for instance in \cite{KloppN-03,KirschM-06}.
Here we require each $(A,h)\in \Omega$ to satisfy additionally to the conditions
in Section  \ref{s-Setting} that $A= X$ and  $ h \equiv 0$ on the diagonal $D$.
The matrix coefficient $h(x,y)$ may be considered as a hopping term between $x$ and $y$
(at least when $h$  is non-negative). In the case that  $h(x,y)\in \{0,1\}$ we obtain a bond-percolation Hamiltonian.
Again, a suitable choice of the measure $\mu$ yields an i.i.d. model.

Note that Dirichlet and Neumann boundary terms as considered in \cite{KirschM-06,AntunovicV-b}
can be incorporated into a potential energy term $\cV_\omega$.
Let us emphasize that the models discussed above do not have necessarily
finite local complexity.

As before denote by $N_{\omega,n}$ the eigenvalue counting functions associated to
a {F\o lner} sequence $I_n$ in $\Gamma$.
Our results from Section \ref{s-Setting} can the be now reformulated as follows:

\begin{enumerate}[(i)]
\item  Let $U_\omega$ be the subspace of $\ell^2(X(\omega))$ spanned by compactly supported eigenfunctions of $H_\omega$
Then, $\nu_H(\{\lambda\})=N_H(\lambda)-\lim\limits_{\epsilon \to 0} N_H(\lambda-\epsilon)$ equals the
equivariant dimension of the subspace $\int^\oplus U_\omega d\mu(\omega)$ of $\int^\oplus_\Omega \ell^2 (X(\omega))d\mu (\omega)$.
\item  If	 $\mu$ is furthermore assumed ergodic,
then  $\lambda$ is a point of discontinuity of $N_H$
if and only if there exist compactly supported eigenfunctions to $H_\omega$ and $\lambda$
for almost every $\omega\in \Omega$. In this case, these eigenfunctions
 actually span the  eigenspace of  $H_\omega$ to the eigenvalue $\lambda$
for almost every $\omega\in \Omega$.
\item
For $\mu$-almost all $\omega$, the distibution function $N_H$ can be approximated by $N_{\omega,n}$ uniformly in the energy variable :
\[
\lim_{n\to\infty}\Big\|\frac{1 }{|\Lambda_n|}N_{\omega,n}- N_H \Big\|_\infty =0.
\]
\end{enumerate}

\section{Percolation on Delone sets}
\label{s-DelonePercolation}

Consider the setting explained in Section \ref{s-Aperiodic}.
In particular, let the space $X$ and the group $\Gamma$ equal $\RR^d$. Let  $\cM\subset\cD$
be a collection of Delone sets of finite local complexity
and $h_0 \colon\RR^d\to\RR$  a bounded, measurable function of compact support satisfying  $h_0(-x)=h_0(x)$.
For each $A \in \cM$ let  $\cE(A)$ consist of pairs of subsets $E_1,E_2$ of $A\times A$ satisfing the following
\begin{enumerate}
 \item $E_1$ and $E_2$ are disjoint,
\item $E_i \cap D=\emptyset$ where as before $D= \{(x,x) \mid x \in \RR^d\}$ and
$(x,y) \in E_i \Rightarrow (y,x) \in E_i $ for $i \in \{1,2\}$.
\end{enumerate}
In other words $E_1,E_2$ is a pair of disjoint sets of edges for the vertex set $A$.
For such $(E_1,E_2)\in \cE(A)$ set
\begin{align*}
 Z_{A,E_1,E_2} := \{ (A,h) \mid & h \colon A\times A \to \RR, h(x,y) \in \{0, h_0(x-y)\}, \\
  &h(x,y) =h_0(x-y) \text{ for all } (x,y) \in E_1, h(x,y) =0 \text{ for all } (x,y) \in E_2\}.
\end{align*}
For $p \in [0,1]$ fixed, define a measure $\mu_{A}$ on the cylinder sets
$Z_{A,E_1,E_2}$  with $E_1$ and $E_2$ finite by setting
\[
 \mu_{A} (Z_{A,E_1,E_2}) := p^{|E_1|} (1-p)^{|E_2|}
\]
and extend it to $\{h \colon A\times A \to \RR\}$ by uniqueness.
Recall that the projection
\[
\pi_1\colon \tcD \to \cD, \quad (A,h) \mapsto A
\]
is measurable. Denote by $\mu_{\cM}$ an invariant probability  measure  on $\cD$ whose support is $\cM$.
Next we define a measure $\mu$  on  $\tcD$. For this purpose we denote the pairs $(A,h)$ by $\omega$
althought the measure $\mu$ and thus its support $\Omega$ are yet to be identified. For a mesurable $ B \subset
\tcD$ we define
\[
\mu(B) = \int_{\cM}\left ( \int_{\pi^{-1} (A) } \chi_B   \ d\mu_A (\omega) \right ) d\mu_{\cM} (A).
\]
Note that if $E_1$ and $E_2$ form a partition of $A\times A$, then $ Z_{A,E_1,E_2}$ contains a single element.
Thus every $\omega = (A,h)$ in the support $\Omega$ of $\mu$ can be identified with such an $ Z_{A,E_1,E_2}$.
The associated operator $H_\omega$ has matrix coefficients
\[
 H_\omega(x,y) = \begin{cases}
                  h_0(x-y) =h_0(y-x) & \text{ if } (x,y) \in E_1 \\
		  0     & \text{ if } (x,y) \in E_2
                 \end{cases}
\]
and defines a bond-percolation Hamiltonian on $\cM$.

Since we are in Euclidean space we can choose $I_n= \Lambda_n$ to be balls or cubes of diameter $n \in \NN$.
\medskip

Again we have the following results:

\begin{enumerate}[(i)]
\item  Let $U_\omega$ be the subspace of $\ell^2(X(\omega))$ spanned by compactly supported solutions  of $(H_\omega-\lambda) u=0$.
Then, $\nu_H(\{\lambda\})=N_H(\lambda)-\lim\limits_{\epsilon \to 0} N_H(\lambda-\epsilon)$ equals the
equivariant dimension of the subspace $\int^\oplus U_\omega d\mu(\omega)$ of $\int^\oplus_\Omega \ell^2 (X(\omega))d\mu (\omega)$.
\item  If $\mu$ is furthermore assumed ergodic,
then  $\lambda$ is a point of discontinuity of $N_H$
if and only if there exists a compactly supported eigenfunctions to $H_\omega$ and $\lambda$
for almost every $\omega\in \Omega$. In this case, these eigenfunctions
 actually span the  eigenspace of  $H_\omega$ to the eigenvalue $\lambda$
for almost every $\omega\in \Omega$.
\item
The distibution function $N_H$ can be approximated by $N_{\omega,n}$ uniformly in the energy variable:
\[
\lim_{n\to\infty}\Big\|\frac{1 }{|\Lambda_n|}N_{\omega,n}- N_H \Big\|_\infty =0
\]
almost surely.
\end{enumerate}

\appendix
\section{Topology on $\tcD$ and compactness}

We discuss the topology and compactness of $\tcD$.

\smallskip

We start with a slightly more general setting. Let $Z$ be a locally compact
space. Denote the set of  measures on $Z$ by $M (Z)$ and the set of
continuous functions with compact support on $Z$ by $C_c (Z)$. The set $M(Z)$
can be embedded into $\prod_{\varphi \in C_c (Z)} \CC$ via
\[
\Psi\colon M(Z)  \longrightarrow \prod_{\varphi \in C_c (Z)} \CC,
\quad \Psi(\mu) = (\varphi \mapsto \mu (\varphi)).
\]
The product topology then induces the initial topology on $M(Z)$, which is called vague
topology. Assume now that $\mathcal{U}$ an open covering of $Z$.
Define for $C>0$ the set
$$ M_{C,\mathcal{U}}:=  M_{C,\mathcal{U}} (Z)  =\{\mu \in M(Z) : \mu (U) \leq C \;\:\mbox{for
all $U\in \mathcal{U}$ }  \}.$$
As $\mathcal{U}$ is an open covering, the support of any $\varphi \in C_c (Z)$ can be covered by finitely many elements of $\mathcal{U}$.  Then, Tychonoff theorem easily yields that $ M_{C,\mathcal{U}}$ is contained in a compact subset of $M (Z)$. As $ M_{C,\mathcal{U}}$
is closed it must then be compact as well.

We now specialize these considerations to the situation outlined in
Section $2$:

Thus, we are given a locally compact metric space $X$.  As before, the set of uniformly discrete subsets of $X$ with minimal
distance $1$ is denoted by $\mathcal{D}$.  Let $Y$ be a compact space and
set $Z=X\times Y$. Let $\mathcal{D}_Y$ be the set of all functions
$f\colon A\longrightarrow Y$ with $A\in \mathcal{D}$.
 We can identify elements of $\mathcal{D}_Y$ with
 measures on $Z$ via
\[
\delta : \mathcal{D}_Y \longrightarrow M(Z), \: (f : A
\longrightarrow Y)\mapsto \sum_{x\in A} \delta_{(x,f(x))}.
\]
This induces the initial topology on $\mathcal{D}_Y$. By a slight abuse of language
we call this topology the vague topology. Consider now the cover
$\mathcal{U}$ of $Z$ consisting of products of the form $B_{1/2} \times Y$
with $B_{1/2}$ an open ball with radius $1/2$ in $X$. Then, $\mathcal{D}_Y$ is
a closed and hence compact subset of $ M_{1,\mathcal{U}} (Z)$.

\medskip

We can even consider a kind of universal $Y$ as follows: Let $\CC^*$
be an arbitrary compactification of $\CC$. Then, $Y:= \mathcal{D}_{\CC^*}$ is a compact space by the preceeding considerations and so is then $\mathcal{D}_Y$.
 For an element $ h \colon A\longrightarrow \CC^*$ in $Y$, we  set $\dom(h) :=A$. It is not hard to see that
\[
\tcD :=\{ f \colon A \longrightarrow Y \mid   \dom(f(x))= A \text{ for all } x \in A \}
\]
is a closed subset of the compact space $\mathcal{D}_Y$.  Hence, $\tcD$ is compact as well.
Any $f\colon A \longrightarrow Y $ satisfying $\dom(f(x))= A$
gives rise to an $h \colon A\times A\longrightarrow \CC^*$ with $h(x,y):= f(x)(y)$.
 We can and will therefore  naturally identify $\tcD$ with the set
\[
\{(A,h)\mid A\in \cD, h\colon A\times A \to \CC^* \}.
\]
\medskip

For the use in the main text let us note  that the topology of $\tcD $ is such that
for any  continuous $\phi\colon X\times X \times \CC^*\to \CC$  of compact support, the map
\[
\tcD \to \CC, \quad
(A,h) \mapsto \sum_{x,y\in A} \phi(x,y,h(x,y))
\]
is continuous. In particular, for a continuous $g \colon X \times X \to \CC^*$ of
compact support, the map $ (A,h) \mapsto \sum_{x,y\in A} g(x,y)h(x,y)$ is continuous.

\medskip


\def\cprime{$'$}\def\polhk#1{\setbox0=\hbox{#1}{\ooalign{\hidewidth
  \lower1.5ex\hbox{`}\hidewidth\crcr\unhbox0}}}

\end{document}